\documentclass[a4paper,10pt]{amsart}
\usepackage{amssymb}
\usepackage{xypic}
\usepackage{mathtools}
\usepackage[dvipdfm]{hyperref}

\allowdisplaybreaks

\DeclareMathOperator{\Hom}{Hom} \DeclareMathOperator{\Ext}{Ext}
\DeclareMathOperator{\Ker}{Ker} \DeclareMathOperator{\op}{op}
\DeclareMathOperator{\cone}{cone} \DeclareMathOperator{\im}{Im}
\DeclareMathOperator{\id}{id} \DeclareMathOperator{\rgld}{r.gl.dim}
\DeclareMathOperator{\ad}{ad} \DeclareMathOperator{\tr}{tr}
\DeclareMathOperator{\RHom}{RHom}
\DeclareMathOperator{\Lotimes}{{}^L\!\otimes}

\newcommand{\diffb}{\mathsf{b}}
\newcommand{\To}{\longrightarrow}
\newcommand{\vep}{\varepsilon}
\newcommand{\vphi}{\varphi}
\newcommand{\D}{\mathbf{D}}
\newcommand{\G}{\mathbf{G}}
\newcommand{\kk}{\Bbbk}
\newcommand\relphantom[1]{\mathrel{\phantom{#1}}}

\numberwithin{equation}{section}
\newtheorem{thm}{Theorem}[section]
\newtheorem{prop}[thm]{Proposition}
\newtheorem{lem}[thm]{Lemma}

\theoremstyle{definition}
\newtheorem{defn}[thm]{Definition}
\newtheorem{rk}[thm]{Remark}

\begin{document}
\setlength{\baselineskip}{1.4em}

\title{Twisted Calabi-Yau property of Ore extensions}
\author{L.-Y.~LIU}
\address{School of Mathematical Sciences, Fudan University, Shanghai 200433, China}
\email{081018015@fudan.edu.cn}
\author{S.-Q.~WANG}
\address{School of Mathematical Sciences, Fudan University, Shanghai 200433, China}
\email{061018004@fudan.edu.cn}
\author{Q.-S.~WU}
\address{School of Mathematical Sciences, Fudan University, Shanghai 200433, China}
\email{qswu@fudan.edu.cn}

\begin{abstract}
Suppose that $E=A[x;\sigma,\delta]$ is an Ore extension with
$\sigma$ an automorphism. It is proved that if $A$ is twisted
Calabi-Yau of dimension $d$, then $E$ is twisted Calabi-Yau of
dimension $d+1$. The relation between their Nakayama automorphisms
is also studied. As an application, the Nakayama automorphisms of a
class of $5$-dimensional Artin-Schelter regular algebras  are given
explicitly.
\end{abstract}


\maketitle

\setcounter{section}{-1}
\section{Introduction}
In the last twenty years, a lot of research appears on
Artin-Schelter regular graded algebras arising from noncommutative
projective algebraic geometry, and on Artin-Schelter regular Hopf
algebras/quantum groups. Brown and Zhang proved that a noetherian
Artin-Schelter regular  Hopf algebra is  rigid Gorenstein
\cite{Brown-Zhang:rdc.hopf.alg},
which is called the twisted Calabi-Yau condition in this paper. Such
a class of algebras is called twisted Calabi-Yau algebra (see
Definition \ref{def:tw.cy}). Van den Bergh duality
\cite{VdB:vdb.duality} holds for any twisted Calabi-Yau algebra. A
noetherian Hopf algebra is Artin-Schelter regular if and only if it
is twisted Calabi-Yau. In the noetherian connected graded case, an
algebra is Artin-Schelter regular if and only if it is graded
twisted Calabi-Yau. Associated to a twisted Calabi-Yau algebra,
there is an automorphism, called Nakayama automorphism in general,
which is unique up to an inner automorphism. A twisted Calabi-Yau
algebra is Calabi-Yau in the sense of Ginzburg \cite{Gin:cy.alg} if
and only if its Nakayama automorphism is inner. Calabi-Yau algebra
is an algebraic structure arising from the geometry of Calabi-Yau
manifolds and mirror symmetry. It has attracted much interest in
recent years.

For any finite-dimensional Lie algebra $\mathfrak{g}$, Yekutieli
constructed the rigid dualizing complex of $U(\mathfrak{g})$
\cite{Yekutieli:rigid.dual.comp.envelop.alg}. In the terminology
now, in fact he proved that $U(\mathfrak{g})$ is Calabi-Yau if and
only if $\tr(\ad x)=0$ for all $x\in\mathfrak{g}$. This result is
generalized to a more general situation --- the PBW deformations of
Koszul Calabi-Yau algebras \cite{Wu-Zhu:pbw.deform.koszul.cy}. The
quantized enveloping algebra of a complex semisimple Lie algebra is
always Calabi-Yau \cite{Chemla:rdc.q.enveloping.alg}. In
\cite{Brown-Zhang:rdc.hopf.alg}, Brown and Zhang also
described the Nakayama automorphism explicitly by using homological
integrals for any noetherian Artin-Schelter regular  Hopf algebras.
Recently, some people are interested in quantum homogeneous spaces,
which are right coideal subalgebras of Hopf algebras satisfying some
additional conditions. One question is to study when quantum
homogeneous spaces are Artin-Schelter regular or twisted Calabi-Yau.
The first named and the third named authors studied the twisted
Calabi-Yau property of the right coideal subalgebras of a quantized
enveloping algebra \cite{LW:tcy.rcs.q.enve.alge}. A class of right
coideal subalgebras of a quantized enveloping algebra can be
obtained by iterated Ore extensions. This motivates us to study the
Nakayama automorphism and the twisted Calabi-Yau property of Ore
extensions in this paper. Ore extension is a noncommutative analogue
of polynomial extension. If $E=A[x;\sigma,\delta]$ is a graded Ore
extension with $\sigma$ an automorphism and $A$ is Artin-Schelter
regular, then so is $E$. This means the twisted Calabi-Yau property
of a connected graded algebra is preserved by (graded) Ore
extensions. It is natural to ask whether Ore extensions preserve
twisted Calabi-Yau property in general situations? The answer is
positive when $\sigma$ is an automorphism.

%

Let $E=A[x;\sigma,\delta]$ be an Ore extension with $\sigma$ an
automorphism. There is a short exact sequence of $E^e$-modules (see
Lemma \ref{lem:ore.ses})
\[0\xrightarrow{\quad}E\otimes_A{}^{\sigma^{-1}}\!E\xrightarrow[\quad]
{\rho}E\otimes_AE\xrightarrow[\quad]{\mu}E\xrightarrow{\quad}0.\]
Then an $E^e$-projective resolution of $E$ can be constructed by
using an $A^e$-projective resolution of $A$. In particular, taking
the bar complex of $A$, the construction is nothing but the
construction given by Guccione-Guccione
\cite{Guccione-Guccione:homology.ore.ext}. Using this construction,
we compute the Hochschild cohomology $H^*(E,E\otimes E)$ and obtain
a family of short exact sequences (Theorem
\ref{thm:ore.hoch.cohomo.ses}).
\begin{thm} \label{Thm1}
  Let $A$ be a projective $\kk$-algebra and $E = A[x; \sigma,\delta]$ be an Ore extension with $\sigma$ an automorphism.
  Suppose that $A$ admits a finitely generated projective resolution as an $A^e$-module.
  Then for any $n\in\mathbb{N}$,
  \[
    0\xrightarrow{} H^{n}(A,E\otimes E)\xrightarrow{}H^{n}(A,E\otimes E^{\sigma^{-1}})
    \xrightarrow {} H^{n+1}(E,E\otimes E)  \xrightarrow{}0
  \]
  is an exact sequence of $E^e$-modules.
\end{thm}

We prove that Ore extensions preserve the twisted Calabi-Yau
property and describe the relation between the Nakayama
automorphisms of $A$ and $E$ (Theorem \ref{thm:nakayama.auto.ore}).
\begin{thm}\label{Thm2}
  Let $A$ be a projective $\kk$-algebra and $E=A[x;\sigma,\delta]$ be an Ore extension with $\sigma$ an automorphism.
  Suppose that $A$ is $\nu$-twisted Calabi-Yau of dimension $d$.
  Then $E$ is twisted Calabi-Yau of dimension $d+1$, and the Nakayama automorphism
  $\nu'$ of $E$
  satisfies that $\nu'|_{A}=\sigma^{-1}\nu$ and $\nu'(x)=ux+b$ for some  $u$, $b\in A$ with $u$ invertible.
\end{thm}

As an application, we focus on a class of Artin-Schelter regular
algebras of dimension $5$ which were investigated in detail by the
second named and the third named authors
\cite{Wang-Wu:5.dim.reg.alg}. Among them, the Nakayama automorphisms
of those may constructed by iterated Ore extensions are given
explicitly.

The paper is organized as follows. In Section \ref{sec:def.ore}, we
recall the definitions of twisted Calabi-Yau algebras and Ore
extensions, and fix some notations. In Section
\ref{sec:hoch.cohomo.ore}, following
\cite{Guccione-Guccione:homology.ore.ext}, we study Hochschild
cohomology on Ore extensions instead of Hochschild homology. Some
exact sequences are obtained and Theorem \ref{Thm1} is proved. In
Section \ref{sec:ore.preserve.tcy}, we prove the main result Theorem
\ref{Thm2}, that is, Ore extensions preserve the twisted Calabi-Yau
property if $\sigma$ is an automorphism. The relation between their
Nakayama automorphisms is also described.  In Section \ref{sec:app},
the main result is applied to multi-parametric quantum affine spaces
and a class of Artin-Schelter regular algebras of dimension $5$
which can be constructed  by iterated Ore extensions.

\section{Preliminaries}\label{sec:def.ore}

\subsection{Twisted Calabi-Yau algebras}
Throughout, $\kk$ is a unital commutative ring and all algebras are
$\kk$-algebras. Unadorned $\otimes$ means $\otimes_{\kk}$ and $\Hom$
means $\Hom_{\kk}$. Suppose that $A$ is an algebra. Let $A^{\op}$ be
the opposite algebra of $A$ and $A^e = A\otimes A^{\op}$ be the
enveloping algebra of $A$. 
The term $A^e$-modules are used for $A$-$A$-bimodules.

For any two $\kk$-modules $M$, $N$, let $\tau_{M,N}\colon M\otimes
N\to N\otimes M$ be the flip map. The subscript is often omitted if
there is no confusion. 
For any $A^e$-module $M$ and any
endomorphisms $\nu$, $\sigma$ of $A$, denote by ${}^\nu\!M^\sigma$
the $A^e$-module whose ground $\kk$-module is $M$ and the action is
given by $a\cdot m\cdot b=\nu(a)m\sigma(b)$ for all $a,b\in A$ and
$m\in M$.  If one of $\nu$ and $\sigma$ is the identity map, then it
is usually omitted.

Suppose that $M$ and $N$ are both $A^e$-modules. It is easy to see
that there are two $A^e$-module structures on $M\otimes N$, one is
called the outer structure defined by $(a\otimes b) \rightharpoonup
(m\otimes n) =am\otimes nb$, and the other is called the inner
structure defined by $ (m\otimes n)\leftharpoondown (a\otimes b)
=ma\otimes bn$, for any $a$, $b\in A$, $m\in M$, $n\in N$. Since
$A^e$ is identified with $A\otimes A$ as a $\kk$-module, $A\otimes A$
endowed with the outer (resp.~inner) structure is nothing but the left
(resp.~right) regular $A^e$-module $A^e$. Hence we often say $A^e$ has the
outer and inner $A^e$-module structures. In the following
definition, the outer structure on $A^e$ is used when computing the
homology $\Ext_{A^e}^{*}(A, A^e)$. Thus $\Ext_{A^e}^{*}(A, A^e)$
admits an $A^e$-module structure induced by the inner one on $A^e$.

\begin{defn}\label{def:tw.cy}
An algebra $A$ is called  \emph{$\nu$-twisted
Calabi-Yau} of dimension $d$ for some automorphism $\nu$ of
$A$ and for some integer $d \geq 0$ if
\begin{enumerate}
\item   $A$ is homologically smooth, that is, as an $A^e$-module,
$A$ has a finitely generated projective resolution of finite length;
\item
$\Ext^i_{A^e}(A, A^e)\cong\begin{cases}0, & \! i \neq d
  \\
  A^\nu, & \! i=d
\end{cases}$\, as $A^e$-modules.
\end{enumerate}

Sometimes condition (2) is called the twisted Calabi-Yau condition.
In this case, $\nu$ is called the \emph{Nakayama automorphism} of
$A$.
\end{defn}

The Nakayama automorphism is unique up to an inner automorphism. A
$\nu$-twisted Calabi-Yau algebra $A$ is Calabi-Yau in the sense of
Ginzburg \cite{Gin:cy.alg} if and only if $\nu$ is an inner
automorphism of $A$.

Graded twisted Calabi-Yau algebras are defined similarly.
Condition (1) is equivalent to that $A$, when viewed as a complex
concentrated in degree $0$, is a compact object in the derived
category $\D(A^e)$ \cite{Neeman:tri.cat}, i.e., the functor
$\Hom_{\D(A^e)}(A, -)$ commutes with arbitrary coproducts.

\subsection{Artin-Schelter regular algebras}
In this subsection, $\kk$ is a field.
\begin{defn}
Suppose that $A$ is an algebra with an augmentation map $\vep\colon
A\to \kk$. Then $A$ is called left \emph{Artin-Schelter regular}
(for short, AS-regular) if
\begin{enumerate}
  \item  $A$ has finite left global dimension $d<\infty,$
  \item $\dim_\kk\Ext^d_A({}_A\kk, {}_AA)=1$ and $\Ext^i_A({}_A\kk, {}_AA)=0$, for all $i\ne
  d$.
\end{enumerate}
\end{defn}

Right AS-regular algebras are defined similarly, and $A$ is called
AS-regular if $A$ is both left and right  AS-regular. A noetherian
Hopf algebra is AS-regular if and only if it is twisted Calabi-Yau.
One direction is proved in \cite[Lemma 5.2 and Proposition
4.5]{Brown-Zhang:rdc.hopf.alg} where they used the term rigid
Gorenstein for twisted Calabi-Yau. The other direction follows from
next lemma, which we can not locate a reference.

\begin{lem} Suppose that $A$ is an algebra with an augmentation map $\vep\colon
A\to \kk$. If $A$ is twisted Calabi-Yau, then $A$ is AS-regular.
\end{lem}

\begin{proof} Since
\begin{align*}
&\relphantom{=}\kk \Lotimes_A A^\nu[-d]  \cong \kk \Lotimes_A \RHom_{A^e}(A, A^e) \cong \RHom_{A^e}(A, A \otimes \kk) \\
 & \cong \RHom_{A^e}(A, \Hom(\kk, A) ) \cong \RHom_{A}( A \otimes_A
\kk, A)\cong \RHom_{A}( \kk, A),
\end{align*}
it follows that $\dim_\kk\Ext^d_A({}_A\kk, {}_AA)=1$ and
$\Ext^i_A({}_A\kk, {}_AA)=0$, for all $i\ne
  d$.
\end{proof}

For a connected graded algebra $A$, $A$ is left AS-regular if and
only if it is right AS-regular. By the same argument as in the above
lemma, $A$ is AS-regular if $A$ is twisted Calabi-Yau. On the other
hand, if $A$ is noetherian AS-regular, then $A$ has a rigid
dualizing complex \cite{VdB:dualizing.complex}, which implies that
$A$ is twisted Calabi-Yau.

\subsection{Ore extensions}
Let $A$ be a $\kk$-algebra, $\sigma$ be an endomorphism of $A$ and
$\delta$ be a $\sigma$-derivation (i.e., $\delta\colon A\to A$ is a
$\kk$-linear map such that
$\delta(ab)=\delta(a)b+\sigma(a)\delta(b)$ for all $a,b\in A$). Then
$\sigma$, $\delta$ uniquely determine a ring extension $E/A$ satisfying
\begin{enumerate}
  \item $E$ is a free left $A$-module with basis $\{1,x,x^2,\ldots\}$,
  \item For any $a\in A$, $xa=\sigma(a)x+\delta(a)$.
\end{enumerate}
The algebra $E$ is denoted by $A[x;\sigma,\delta]$ and is called the
\emph{Ore extension} of $A$ associated to $\sigma$ and $\delta$. For
graded algebras, \emph{graded Ore extensions} are defined similarly.
However, the Koszul sign convention does not apply in this context.

If $\sigma$ is the identity map, $A[x;\sigma,\delta]$ is often simply written as $A[x;\delta]$;
and if $\delta=0$, as $A[x;\sigma]$. The polynomial extension $A[x]$ is a special Ore extension.

If $\sigma$ is an automorphism, then $\{1,x,x^2,\ldots\}$ is also a
basis for $E$ as a free right $A$-module. In this case,
$Ax^k\subseteq\sum_{i=0}^kx^iA$ and $x^l A\subseteq\sum_{j=0}^lAx^j$
for any $k$, $l\in\mathbb{N}$.  Let $p^n_i$ be the $\kk$-linear map
which is the sum of all the compositions $\sigma_1 \sigma_2 \cdots
\sigma_n$ with $\sigma_j$ being $\sigma$ or $\delta$, and $\sigma$
appearing $i$ times in each composition. Then for any $a
\in A$ and $n \ge 1$,
\begin{equation}\label{eq:ore.right}
  x^na=\sum_{i=0}^np^n_i(a)x^i.
\end{equation}
Similarly, let $q^n_i$ be the $\kk$-linear map which is the sum of
all the compositions $\sigma_1 \sigma_2 \cdots \sigma_n$ with
$\sigma_j$ being $\sigma^{-1}$ or $-\delta \sigma^{-1}$,
and $\sigma^{-1}$ appearing $i$ times in each composition. Then for
any $a \in A$ and $n \ge 1$,
\begin{equation}\label{eq:ore.left}
  ax^n=\sum_{i=0}^nx^iq^n_i(a).
\end{equation}


Many ring-theoretic and homological properties are preserved by Ore
extensions under certain conditions. We list some of them as
follows.
\begin{itemize}
  \item If $A$ is an integral domain and $\sigma$ is injective, then $E$ is an integral domain.
  \item If $A$ is a prime ring and $\sigma$ is an automorphism, then $E$ is a prime ring.
  \item If $A$ has finite right global dimension and $\sigma$ is an automorphism,
  then $E$ has finite right global dimension, in fact,
      \[\rgld A\le \rgld E\le \rgld A+1.\]
  \item If $\kk$ is a noetherian ring, $A$ is (strongly) right noetherian and $\sigma$ is an automorphism,
  then $E$ is (strongly) right noetherian.
\end{itemize}
For the details and other properties of Ore extensions, we refer to
\cite{McConnell-Robson:noncomm.noeth.ring},
\cite{Goodearl-Warfield:noncommu.noeth.ring}, \cite{Artin-Small-Zhang:stron.noeth.alg}, 
etc.

Here are some examples of iterated Ore extensions: multi-parameter
quantum affine $n$-spaces $\mathcal{O}_{\mathbf{q}}(\kk^n)$, Weyl
algebras $A_n(\kk)$, enveloping algebras $U(\mathfrak{g})$ of
finite-dimensional nilpotent Lie algebras $\mathfrak{g}$, the Borel
part of quantized enveloping algebras $U_q(\mathfrak{g})$ of complex
semisimple Lie algebras $\mathfrak{g}$, and some classes of
AS-regular algebras.

\subsection{Notations}
We fix some notations about complexes and graded modules.

Suppose that $(P_\cdot,d)$ is a chain complex. The $l$-shift of
$P_\cdot$, denoted by $P[l]_\cdot$, is defined by $P[l]_n=P_{n-l}$
and $d[l]_n=(-1)^ld_{n-l}$. If $(P'_\cdot,d')$ is another chain
complex and $f\colon P_\cdot\to P'_\cdot$ is a morphism of
complexes, the mapping cone of $f$, denoted by $\cone(f)$, is
defined by $\cone(f)_n=P_{n-1}\oplus P'_n$ and the differential
sending $(p,p')\in\cone(f)_n$ to
$(-d_{n-1}(p),d'_n(p')-f_{n-1}(p))$. Dually, suppose that
$(Q^\cdot,d)$, $(Q'^\cdot,d')$ are cochain complexes and $g\colon
Q^\cdot\to Q'^\cdot$ is a morphism of complexes. The $l$-shift of
$Q^\cdot$, denoted by $Q[l]^\cdot$, is defined by $Q[l]^n=Q^{n+l}$
and $d[l]^n=(-1)^ld^{n+l}$. The mapping cone $\cone(g)$, is defined
by $\cone(g)^n=Q^{n+1}\oplus Q'^n$ and the differential sending
$(q,q')\in\cone(g)^n$ to $(-d^{n+1}(q),d'^n(q')+g^{n+1}(q))$. If
$f\colon P_\cdot\to P'_\cdot$  is a morphism of
$A$-module complexes and ${}_AM$ is an $A$-module, then $\cone (\Hom_A (f,
M)) \cong \Hom_A (\cone(f), M)[1].$

For any graded $A$-module $M$, the $n$-shift $M(n)$ of $M$, is
defined by $M(n)_i = M_{n+i}$.

We mainly refer to \cite{Loday:cyclic.homology} for Hochschild
homology and cohomology.

\section{Hochschild cohomology on Ore extensions}\label{sec:hoch.cohomo.ore}

%
%
%
%
We investigate the Hochschild cohomology on Ore extensions
in this section. From now on, $\sigma$ is always required to be an
automorphism.

\begin{lem}\label{lem:ore.ses}
  Let $A$ be an algebra and $E=A[x;\sigma,\delta]$ be an Ore extension. Then the sequence of $E^e$-modules
  \begin{equation}\label{eq:ore.ses}
    0\xrightarrow{\quad}E\otimes_A{}^{\sigma^{-1}}\!E\xrightarrow[\quad]{\rho}E\otimes_AE\xrightarrow[\quad]{\mu}
    E\xrightarrow{\quad}0
  \end{equation}
  is exact, where $\rho(e\otimes e')=ex\otimes e'-e\otimes xe'$ and $\mu$ is the multiplication.
\end{lem}
\begin{proof}
  First of all, $\rho$ is well-defined since
  \begin{align*}
    \rho(1\otimes\sigma^{-1}(a))&=x\otimes\sigma^{-1}(a)-1\otimes x\sigma^{-1}(a)\\
    &=x\otimes\sigma^{-1}(a)-1\otimes ax-1\otimes\delta\sigma^{-1}(a)\\
    &=x\sigma^{-1}(a)\otimes 1-a\otimes x-\delta\sigma^{-1}(a)\otimes 1\\
    &=ax\otimes 1-a\otimes x\\
    &=\rho(a\otimes 1).
  \end{align*}

  Suppose $\sum_{i=0}^n x^i\otimes e_i \in\Ker\rho$.
  Then $\sum_{i=0}^n x^{i+1}\otimes e_i-\sum_{i=0}^n x^i\otimes xe_i=0$.
  Note that $x^{n+1}\otimes e_n$ is the unique term containing $x^{n+1}$ as the first tensor factor.
  It follows that $e_n=0$ and so $\sum_{i=0}^n x^i\otimes e_i=0$. Thus $\rho$ is injective.

  Now suppose $\sum_{i=0}^n x^i\otimes e'_i \in\Ker\mu$ with $e'_n\neq 0$. Then
  \begin{align*}
    \sum_{i=0}^n x^i\otimes e'_i&=1 \otimes e'_0+\sum_{i=1}^n x^i\otimes e'_i\\
    &=1 \otimes e'_0+\sum_{i=1}^n x^i\otimes e'_i-\sum_{i=1}^n x^{i-1}\otimes xe'_i+\sum_{i=1}^n x^{i-1}\otimes xe'_i\\
    &=1 \otimes e'_0+\rho\bigg(\sum_{i=1}^n x^{i-1}\otimes e'_i\bigg)+\sum_{i=0}^{n-1} x^{i}\otimes xe'_{i+1}\\
    &=\sum_{i=0}^{n-1} x^i\otimes e''_i\pmod {\im\rho}
  \end{align*}
  where $e''_i\in E$ and $e''_{n-1}\neq 0$. By induction on $n$, we obtain $\Ker\mu=\im \rho$.

  Therefore, the sequence \eqref{eq:ore.ses} is exact.
\end{proof}
\begin{rk}
  The graded version of Lemma \ref{lem:ore.ses} is also true. If $\deg(x)=l$,
  the short exact sequence \eqref{eq:ore.ses} should be modified by
  \[
    0\xrightarrow{\quad}E\otimes_A{}^{\sigma^{-1}}\!E(-l)\xrightarrow[\quad]{\rho}E\otimes_AE\xrightarrow[\quad]{\mu}
    E\xrightarrow{\quad}0.
  \]
\end{rk}
For any $A^e$-projective resolution $P_{\cdot}$ of $A$ with an
augmentation map $\vep$,
$E\otimes_AP_{\cdot}\otimes_A{}^{\sigma^{-1}}\!E$ and
$E\otimes_AP_{\cdot}\otimes_AE$ are $E^e$-projective resolutions of
$E\otimes_A{}{}^{\sigma^{-1}}\!E$ and $E\otimes_AE$ respectively. By
the Comparison lemma, $\rho$ can be lifted to a morphism of
$E^e$-module complexes from
$E\otimes_AP_{\cdot}\otimes_A{}{}^{\sigma^{-1}}\!E$ to
$E\otimes_AP_{\cdot}\otimes_AE$, say $\psi$. Then $\cone(\psi)$ is
an $E^e$-projective resolution of $E$ via
$\mu(\id_E\otimes\,\vep\otimes\id_E)$.


Now we start to look at the Hochschild cohomology. Let $P_{\cdot}$
be the bar complex of $A$,
\begin{equation*}
    0 \xleftarrow{\quad} A^{\otimes 2} \xleftarrow[\quad]{b'} A^{\otimes
    3}\xleftarrow[\quad]{b'} \cdots \xleftarrow[\quad]{b'} A^{\otimes n+1}
   \xleftarrow[\quad]{b'} A^{\otimes n+2} \xleftarrow[\quad]{b'} \cdots
  \end{equation*}
where $b'\colon A^{\otimes n+2} \to A^{\otimes n+1}$ is the map
   \[b'(a_0\otimes\cdots\otimes a_{n+1})=\sum_{i=0}^{n}(-1)^i a_0\otimes\cdots\otimes a_ia_{i+1} \cdots\otimes
   a_{n+1}.\]

 A lifting map of $\rho$ is constructed in
\cite{Guccione-Guccione:homology.ore.ext} as follows.

The two complexes $E\otimes_AP_{\cdot}\otimes_A{}^{\sigma^{-1}}\!E$
and $E\otimes_AP_{\cdot}\otimes_AE$ are $(E\otimes A^{\otimes
*}\otimes {}^{\sigma^{-1}}\!E,b'_{1,*})$ and $(E\otimes A^{\otimes
*}\otimes E,b'_{0,*})$, respectively, where the differentials are
\begin{align*}
  b'_{0,n}(a_0\otimes\cdots\otimes a_{n+1}) &= \sum^n_{i=0}(-1)^ia_0 \otimes\cdots\otimes a_ia_{i+1}
  \otimes\cdots\otimes a_{n+1},\\
  b'_{1,n}(a_0\otimes\cdots\otimes a_{n+1}) &= \sum^{n-1}_{i=0}(-1)^ia_0 \otimes
    \cdots\otimes a_ia_{i+1} \otimes\cdots\otimes
    a_{n+1}\\
    &\relphantom{=}{}+(-1)^na_0 \otimes\cdots\otimes a_{n-1}\otimes \sigma^{-1}(a_{n}) a_{n+1}.
\end{align*}
The lifting map $\{\psi'_n\colon E\otimes A^{\otimes n}\otimes {}^{\sigma^{-1}}\!E\to E\otimes
A^{\otimes n}\otimes E\}_{n\in\mathbb{N}}$ is defined by
\begin{align*}
  \psi'_{n}&(1\otimes a_1\otimes\cdots\otimes a_{n}\otimes 1)\\
  &= x\otimes \sigma^{-1}(a_1)\otimes\cdots\otimes \sigma^{-1}(a_{n})\otimes 1-1\otimes a_1\otimes\cdots
  \otimes a_{n}\otimes x\\
  &\relphantom{=}{}-\sum^{n}_{j=1}1\otimes a_1 \otimes\cdots\otimes a_{j-1}\otimes\delta\sigma^{-1}(a_{j})\otimes
  \sigma^{-1}(a_{j+1})\otimes\cdots\otimes
  \sigma^{-1}(a_{n}) \otimes 1.
\end{align*}

By the above argument, we have
\begin{lem}[{\cite[Propositions 1.1 and
1.2]{Guccione-Guccione:homology.ore.ext}}] \label{GG}
  Let $A$ be an algebra and $E = A[x; \sigma,\delta]$ be an Ore extension. Then
  \begin{equation}\label{cd:ore-hochschild}
    \xymatrix@C=1.0cm{
    E\otimes {}^{\sigma^{-1}}\!E \ar[d]^{\psi'_0}
                & E\otimes A\otimes {}^{\sigma^{-1}}\!E \ar[d]^{\psi'_1} \ar[l]_{b'_{1,1}} &
                E\otimes A^{\otimes 2}\otimes {}^{\sigma^{-1}}\!E \ar[d]^{\psi'_2} \ar[l]_{b'_{1,2}}
                &  \cdots \ar[l]_(.3){b'_{1,3}} \\
    E\otimes E
                & E\otimes A\otimes E  \ar[l]_{b'_{0,1}} & E\otimes A^{\otimes 2}\otimes E  \ar[l]_{b'_{0,2}}
                          &  \cdots \ar[l]_(.3){b'_{0,3}}      }
  \end{equation}
  is a commutative diagram of $E^e$-modules, and
  \begin{equation}\label{eq:resolution.ore.ext}
    \cone(\psi')\xrightarrow[\quad]{\mu} E\xrightarrow[\quad]{}0
  \end{equation}
  is an exact sequence. If further, $A$ is flat (resp.~projective) over $\kk$,
  then \eqref{eq:resolution.ore.ext} is a flat (resp.~projective) resolution of $E$ as an $E^e$-module.
\end{lem}

In the following statements, we sometimes write
$f(a_1\otimes\cdots\otimes a_{n})$ as $f(a_1, \dots, a_{n})$ for
convenience.

Let $M$ be an $E^e$-module. Applying $\Hom_{E^e}(-,M)$ to \eqref{cd:ore-hochschild},
we have the following commutative diagram
\begin{equation}\label{cd:hoch.cohomo.ore.ext}
  \xymatrix@C=1.0cm{
  \Hom(\kk,M^{\sigma^{-1}}) \ar[r]^{b^{1,0}}
                & \Hom(A,M^{\sigma^{-1}}) \ar[r]^{b^{1,1}}
                & \Hom(A^{\otimes 2},M^{\sigma^{-1}}) \ar[r]^(.7){b^{1,2}}  &  \cdots  \\
  \Hom(\kk,M) \ar[u]_{\theta^0} \ar[r]^{b^{0,0}}
                & \Hom(A,M)  \ar[u]_{\theta^1}\ar[r]^{b^{0,1}}  &  \Hom(A^{\otimes 2},M)
          \ar[u]_{\theta^2}\ar[r]^(.7){b^{0,2}}  &  \cdots         }
\end{equation}
where the maps are given by,  for any $f\in\Hom(A^{\otimes n},M)$,
$\tilde{f}\in\Hom(A^{\otimes n},M^{\sigma^{-1}})$,
\begin{align*}
    &\begin{aligned}
      b^{0,n}(f)(a_1,\dots, a_{n+1}) &= a_1f(a_2,\dots, a_{n+1})+\sum^{n}_{i=1}(-1)^{i}f(a_1,\dots,
      a_ia_{i+1},\dots, a_{n+1})\\
      &\relphantom{=}{}+(-1)^{n+1}f(a_1,\dots, a_{n})a_{n+1},
    \end{aligned}\\
    &\begin{aligned}
      b^{1,n}(\tilde{f})(a_1,\dots, a_{n+1}) &= a_1\tilde{f}(a_1,\dots, a_{n+1})+
      \sum^{n}_{i=1}(-1)^{i}\tilde{f}(a_1,\dots,
      a_ia_{i+1},\dots, a_{n+1})\\
      &\relphantom{=}{}+(-1)^{n+1}\tilde{f}(a_1,\dots, a_{n})\sigma^{-1}(a_{n+1}),
    \end{aligned}\\
    &\begin{aligned}
      \theta^{n}(f)(a_1,\dots, a_{n}) &= xf(\sigma^{-1}(a_1),\dots, \sigma^{-1}(a_{n}))
      -f(a_1,\dots, a_{n})x\\
      &\relphantom{=}{}-\sum^{n}_{j=1}f(a_1,\dots, a_{j-1},\delta\sigma^{-1}(a_{j}),\sigma^{-1}(a_{j+1}),\dots,
      \sigma^{-1}(a_{n})).
    \end{aligned}
\end{align*}

Obviously, when $M$ is viewed as an $A^e$-module, the two rows in
the diagram \eqref{cd:hoch.cohomo.ore.ext} is the Hochschild complex
$C^*(A,M^{\sigma^{-1}})$ and $C^*(A,M)$. In general, for any
$A^e$-module $M$, the differentials of $C^*(A,M)$ and
$C^*(A,M^{\sigma^{-1}})$ are denoted by $\diffb$ and
$\diffb_{\sigma^{-1}}$ respectively, if there is no confusion. On
the other hand, by Lemma \ref{GG}, we can compute $H^{n}(E,M)$ by
using $\cone(\psi')$ or $\cone(\theta)$.

\begin{lem}\label{lem:hoch.cohomo.ore.ext.cone}
  Let $A$ be a projective $\kk$-algebra and $E = A[x; \sigma,\delta]$ be an Ore extension
  and let $M$ be an $E^e$-module. For any $n\in\mathbb{N}$, $H^{n}(E,M)\cong H^{n-1}(\cone(\theta))$.
\end{lem}
\begin{proof}
  By \eqref{eq:resolution.ore.ext} and \eqref{cd:hoch.cohomo.ore.ext},
  \[
    H^n(E,M)=H^n(\Hom_{E^e}(\cone(\psi'),M))\cong H^n(\cone(\theta)[-1])=H^{n-1}(\cone(\theta)).
  \]
\end{proof}

Now let $M=E\otimes E$. By the definition of mapping cones, there is a short
exact sequence of $E^e$-module complexes
\[
  0\xrightarrow{\quad}C^*(A,E\otimes E^{\sigma^{-1}})\xrightarrow{\quad}\cone(\theta)
  \xrightarrow{\quad}C^*(A,E\otimes E)[1]\xrightarrow{\quad}0,
\]
where the $E^e$-module structure on each complex is induced by the
inner structure on $E\otimes E$.  It follows that
\begin{equation*}
\begin{split}
  \cdots&\xrightarrow{\quad}H^{n-1}(C^*(A,E\otimes E^{\sigma^{-1}}))\xrightarrow{\quad}H^{n-1}(\cone(\theta))
  \xrightarrow{\quad}H^{n-1}(C^*(A,E\otimes E)[1])\\
  &\xrightarrow[\quad]{\partial}H^{n}(C^*(A,E\otimes E^{\sigma^{-1}}))
  \xrightarrow{\quad}H^{n}(\cone(\theta))\xrightarrow{\quad}H^{n}(C^*(A,E\otimes E)[1])\xrightarrow{\quad}\cdots
\end{split}
\end{equation*}
is an exact sequence of $E^e$-modules. By Lemma \ref{lem:hoch.cohomo.ore.ext.cone}, the above sequence becomes
\begin{equation}\label{eq:long.e.s.hoch.cohomo.ore.ext}
\begin{split}
  \cdots&\xrightarrow{\quad}H^{n-1}(A,E\otimes E^{\sigma^{-1}})\xrightarrow{\quad}H^{n}(E,E\otimes E)\xrightarrow{\quad}
  H^{n}(A,E\otimes E)\\
  &\xrightarrow[\quad]{\partial}H^{n}(A,E\otimes E^{\sigma^{-1}})\xrightarrow{\quad}H^{n+1}(E,E\otimes E)
  \xrightarrow{\quad}H^{n+1}(A,E\otimes E)
  \xrightarrow{\quad}\cdots,
\end{split}
\end{equation}
where the connecting homomorphism $\partial =H^n(\theta)$.

Since, as $A^e$-modules, $E\otimes E\cong A\otimes A\otimes
\kk[x]^{\otimes 2}, \, ax^l \otimes x^kb \mapsto a \otimes b \otimes
x^l\otimes x^k$,  and similarly
\[E\otimes E^{\sigma^{-1}}\cong A\otimes A^{\sigma^{-1}}\otimes
\kk[x]^{\otimes 2}\cong A\otimes{}^{\sigma}\!A\otimes
\kk[x]^{\otimes 2},\]
there exist two canonical morphisms of
$\kk$-module complexes
\begin{gather*}
  C^*(A,A\otimes A^{\sigma^{-1}})\otimes \kk[x]^{\otimes 2}\to
  C^*(A,E\otimes E^{\sigma^{-1}}),\\
  C^*(A,A\otimes A)\otimes \kk[x]^{\otimes 2}\to C^*(A,E\otimes E).
\end{gather*}
where the differentials of the left two complexes are
$\diffb_{\sigma^{-1}}\otimes\id^{\otimes 2}$,
$\diffb\otimes\id^{\otimes 2}$, respectively.

We hope to equip the left two complexes with suitable $E^e$-module
structures such that the above are morphisms of
$E^e$-module complexes. To this end, for any $\tilde{f}\in C^n(A,A\otimes
A^{\sigma^{-1}})$, define
\begin{align*}
&x\cdot(\tilde{f}\otimes x^{l}\!\otimes x^k)=\tilde{f}\otimes x^{l}\!\otimes x^{k+1},&& \forall\, k,l\in\mathbb{N},\\
&a\cdot(\tilde{f}\otimes x^{l}\!\otimes x^k)
=\sum_{i=0}^kq^k_i(a)\cdot\tilde{f}\otimes x^{l}\otimes x^i,&& \forall\,a\in A,\\
&(\tilde{f}\otimes x^{l}\!\otimes x^k)\cdot x=\tilde{f}\otimes x^{l+1}\!\otimes x^k,\\
&(\tilde{f}\otimes x^{l}\!\otimes x^k)\cdot a
=\sum_{i=0}^l\tilde{f}\cdot p^l_i(a)\otimes x^{i}\otimes x^k,
\end{align*}
where $p^l_i$  and $q^k_i$ are defined in \eqref{eq:ore.right} and
\eqref{eq:ore.left} respectively, the actions
$q^k_i(a)\cdot\tilde{f}$ and $\tilde{f}\cdot p^l_i(a)$ are induced
from the inner structure on $A\otimes A^{\sigma^{-1}}$.

This makes $C^*(A,A\otimes A^{\sigma^{-1}})\otimes \kk[x]^{\otimes
2}$ be a complex of $E^e$-modules and similarly for $C^*(A,A\otimes
A)\otimes \kk[x]^{\otimes 2}$.

\begin{lem}
  Suppose that $A$ is a flat $\kk$-algebra and $E = A[x; \sigma,\delta]$ is an Ore extension.
  Then there exists a morphism of $E^e$-module complexes
  \[\eta\colon C^*(A,A\otimes A)\otimes \kk[x]^{\otimes 2}\to C^*(A,A\otimes A^{\sigma^{-1}})\otimes \kk[x]^{\otimes 2}\]
  such that the following diagram is commutative,
  \begin{equation}\label{cd:quasiiso.ore.ext}
    \xymatrix{
     C^*(A,A\otimes A^{\sigma^{-1}})\otimes \kk[x]^{\otimes 2} \ar[r]
        &   C^*(A,E\otimes E^{\sigma^{-1}})       \\
     C^*(A,A\otimes A)\otimes \kk[x]^{\otimes 2} \ar[u]_{\eta}\ar[r]
        & C^*(A,E\otimes E) \ar[u]_{\theta}.
     }
  \end{equation}
\end{lem}
\begin{proof}
  For any $y\in A\otimes A^{\sigma^{-1}}$, we use Sweedler's notation $y=\sum y'\otimes y''$.
  For any $f \otimes x^l \otimes x^k \in C^n(A, A\otimes A)\otimes \kk[x]^{\otimes
  2}$, let
  \[[f, x^l \otimes x^k]\colon (a_1,\dots, a_{n}) \mapsto
  \sum f(a_1,\dots, a_{n})'x^l \otimes x^kf(a_1,\dots, a_{n})''\]
  be the corresponding element in $C^n(A,E\otimes E)$.  Then
\begin{align*}
    &\relphantom{=}\theta^{n}([f, x^l \otimes x^k ])(a_1,\dots, a_{n})\\
    &=x([f, x^l \otimes x^k ](\sigma^{-1}(a_1),\dots,
    \sigma^{-1}(a_{n}))) -([f, x^l \otimes x^k ](a_1,\dots, a_{n}))x\\
    &\relphantom{=}{}-\sum^{n}_{j=1}[f,x^l \otimes x^k ](a_1,\dots, a_{j-1},\delta\sigma^{-1}(a_{j}),
    \sigma^{-1}(a_{j+1}),\dots,  \sigma^{-1}(a_{n}))\\
    &=\sum xf(\sigma^{-1}(a_1),\dots, \sigma^{-1}(a_{n}))'x^l\otimes
    x^k f(\sigma^{-1}(a_1),\dots, \sigma^{-1}(a_{n}))''  \\
    &\relphantom{=}{}-\sum f(a_1,\dots, a_{n})'x^l\otimes x^{k}f(a_1,\dots, a_{n})''x \\
    &\relphantom{=}{}-\sum^{n}_{j=1}[f,x^l \otimes x^k ](a_1,\dots, a_{j-1},\delta\sigma^{-1}(a_{j}),
    \sigma^{-1}(a_{j+1}),\dots,  \sigma^{-1}(a_{n})) \\
    &=\sum \sigma(f(\sigma^{-1}(a_1),\dots, \sigma^{-1}(a_{n}))')x^{l+1} \otimes x^k f(\sigma^{-1}(a_1),\dots,
    \sigma^{-1}(a_{n}))''\\
    &\relphantom{=}{}+\sum \delta(f(\sigma^{-1}(a_1),\dots, \sigma^{-1}(a_{n}))')x^l
    \otimes x^k f(\sigma^{-1}(a_1),\dots, \sigma^{-1}(a_{n}))''  \\
    &\relphantom{=}{}-\sum f(a_1,\dots, a_{n})'x^l\otimes x^{k+1}\sigma^{-1}(f(a_1,\dots, a_{n})'')  \\
    &\relphantom{=}{}+\sum f(a_1,\dots, a_{n})'x^l\otimes x^k \delta\sigma^{-1}(f(a_1,\dots, a_{n})'') \\
    &\relphantom{=}{}-\sum^{n}_{j=1}[f,x^l \otimes x^k ](a_1,\dots, a_{j-1},\delta\sigma^{-1}(a_{j}),
    \sigma^{-1}(a_{j+1}),\dots,  \sigma^{-1}(a_{n})) \\
    &=\big[(\sigma\otimes\id)f(\sigma^{-1})^{\otimes n},x^{l+1} \otimes x^k\big](a_1,\dots, a_{n}) \\
    &\relphantom{=}{}-\big[(\id\otimes\,\sigma^{-1})f,x^l \otimes x^{k+1}\big](a_1,\dots, a_{n})\\
    &\relphantom{=}{}+\big[(\delta\otimes\id)f(\sigma^{-1})^{\otimes n},x^l \otimes x^k\big](a_1,\dots,a_{n}) \\
    &\relphantom{=}{}+\big[(\id\otimes\,\delta\sigma^{-1})f,x^l \otimes x^k\big](a_1,\dots, a_{n})\\
    &\relphantom{=}{}-\sum^{n}_{j=1}\big[f(\id^{\otimes j-1}\otimes\,\delta\sigma^{-1}
    \otimes(\sigma^{-1})^{\otimes n-j}),x^l \otimes x^k\big](a_1,\dots, a_{n}).
  \end{align*}

  Thus $\eta$ can be defined as follows, so that the diagram \eqref{cd:quasiiso.ore.ext} is commutative.
  For any $n\in\mathbb{N}$ and $f\in C^n(A,A\otimes A)$,
  \begin{equation}\label{eq:eta}
    \eta^n(f\otimes x^{l}\otimes x^{k})=f_1\otimes x^{l+1}\otimes x^{k}-f_2\otimes x^{l}
    \otimes x^{k+1}+f_3\otimes x^{l}\otimes x^{k}
  \end{equation}
  with
  \begin{align}
    f_1&\coloneqq(\sigma\otimes\id)f(\sigma^{-1})^{\otimes n}\label{eq:f1}\\
    f_2&\coloneqq(\id\otimes\,\sigma^{-1})f\label{eq:f2}\\
    f_3&\coloneqq(\delta\otimes\id)f(\sigma^{-1})^{\otimes n}
    +(\id\otimes\,\delta\sigma^{-1})f\label{eq:f3}\\
    &\relphantom{\coloneqq}{}-\sum^{n}_{j=1}f(\id^{\otimes j-1}\otimes\,\delta\sigma^{-1}
    \otimes(\sigma^{-1})^{\otimes n-j}).\notag
  \end{align}

  It remains to check that $\eta^n$ is $E^e$-linear for all $n$. In fact,
  it is obvious that $\eta^n(x\cdot(f\otimes x^l\otimes x^k)\cdot x)=x\cdot\eta^n(f\otimes x^l\otimes x^k)\cdot x$.
  Thus it suffices to show
  \begin{align}
    \eta^n(a\cdot(f\otimes 1\otimes 1))=a\cdot\eta^n(f\otimes 1\otimes 1),\label{eq:left.linear}\\
    \eta^n((f\otimes 1\otimes 1)\cdot a)=\eta^n(f\otimes 1\otimes 1)\cdot a.\label{eq:right.linear}
  \end{align}

  By the definition of $\eta$,
  \begin{align*}
    &\relphantom{=}\eta^n(a\cdot(f\otimes 1\otimes 1))=\eta^n(a\cdot f\otimes 1\otimes 1)\\
    &=(a\cdot f)_1\otimes x\otimes 1-(a\cdot f)_2\otimes 1\otimes x+(a\cdot f)_3\otimes 1\otimes 1,
  \end{align*}
    and
  \begin{align*}
    &\relphantom{=}a\cdot\eta^n(f\otimes 1\otimes 1)\\
    &=a\cdot(f_1\otimes x\otimes 1-f_2\otimes 1\otimes x+f_3\otimes 1\otimes 1)\\
    &=a\cdot f_1\otimes x\otimes 1 -\sigma^{-1}(a)\cdot f_2\otimes 1\otimes x
    +\delta\sigma^{-1}(a)\cdot f_2\otimes 1\otimes 1 +a\cdot f_3\otimes 1\otimes 1\\
    &=a\cdot f_1\otimes x\otimes 1-\sigma^{-1}(a)\cdot f_2\otimes 1\otimes x
    +(\delta\sigma^{-1}(a)\cdot f_2+a\cdot f_3)\otimes 1\otimes 1.
  \end{align*}
  It is easy to verify $(a\cdot f)_1=a\cdot f_1$,
  $(a\cdot f)_2=\sigma^{-1}(a)\cdot f_2$. And
  \begin{align*}
    &\relphantom{=}(a\cdot f)_3(a_1,\dots,a_n)\\
    &=(\delta\otimes\id)(a\cdot f)(\sigma^{-1}(a_1),\dots,\sigma^{-1}(a_{n}))
    +(\id\otimes\,\delta\sigma^{-1})(a\cdot f)(a_1,\dots, a_{n})\\
    &\relphantom{=}{}-\sum^{n}_{j=1}(a\cdot f)(a_1,\dots, a_{j-1},\delta\sigma^{-1}(a_{j}),
    \sigma^{-1}(a_{j+1}),\dots,  \sigma^{-1}(a_{n}))\\
    &=\sum\delta(f(\sigma^{-1}(a_1),\dots,\sigma^{-1}(a_{n}))')\otimes af(\sigma^{-1}(a_1),\dots,\sigma^{-1}(a_{n}))''\\
    &\relphantom{=}{}+\sum f(a_1,\dots, a_{n})'\otimes\delta\sigma^{-1}(af(a_1,\dots, a_{n})'')\\
    &\relphantom{=}{}-\sum^{n}_{j=1}(a\cdot f)(a_1,\dots, a_{j-1},\delta\sigma^{-1}(a_{j}),
    \sigma^{-1}(a_{j+1}),\dots, \sigma^{-1}(a_{n}))\\
    &=\sum\delta(f(\sigma^{-1}(a_1),\dots,\sigma^{-1}(a_{n}))')\otimes af(\sigma^{-1}(a_1),\dots,\sigma^{-1}(a_{n}))''\\
    &\relphantom{=}{}+\sum f(a_1,\dots, a_{n})'\otimes\delta\sigma^{-1}(a)\sigma^{-1}(f(a_1,\dots, a_{n})'')\\
    &\relphantom{=}{}+\sum f(a_1,\dots, a_{n})'\otimes a\delta\sigma^{-1}(f(a_1,\dots, a_{n})'')\\
    &\relphantom{=}{}-\sum^{n}_{j=1}(a\cdot f)(a_1,\dots, a_{j-1},
    \delta\sigma^{-1}(a_{j}),\sigma^{-1}(a_{j+1}),\dots, \sigma^{-1}(a_{n}))\\
    &=a \cdot (\delta\otimes\id)f(\sigma^{-1}(a_1),\dots,\sigma^{-1}(a_{n}))\\
    &\relphantom{=}{}+ \delta\sigma^{-1}(a) \cdot (\id\otimes\,\sigma^{-1})f(a_1,\dots,
    a_{n})\\
    &\relphantom{=}{}+a \cdot (\id\otimes\,\delta\sigma^{-1})f(a_1,\dots,
    a_{n})\\
    &\relphantom{=}{}-\sum^{n}_{j=1}(a\cdot f)(a_1,\dots, a_{j-1},\delta\sigma^{-1}(a_{j}),
    \sigma^{-1}(a_{j+1}),\dots, \sigma^{-1}(a_{n}))\\
    &=(\delta\sigma^{-1}(a)\cdot f_2+a\cdot f_3)(a_1\dots,a_n).
  \end{align*}

  Thus \eqref{eq:left.linear} holds and \eqref{eq:right.linear} can be checked in a similar way.
  Therefore, $\eta$ is constructed as desired.
\end{proof}

\begin{lem}\label{lem:f1.f2.f3}
  Suppose that $A$ is a flat $\kk$-algebra and $E = A[x; \sigma,\delta]$ is an Ore extension.
  Let $f\in C^{n}(A,A\otimes A)$ ($n\in\mathbb{N}$)  and $f_1$, $f_2$, $f_3$ be given
  by \eqref{eq:f1}, \eqref{eq:f2}, \eqref{eq:f3}. The following are equivalent:
    \begin{enumerate}
      \item $f$ is a cocycle (resp.~coboundary) in $C^{n}(A,A\otimes A)$,
      \item $f_1$ is a cocycle (resp.~coboundary) in $C^{n}(A,A\otimes A^{\sigma^{-1}})$,
      \item $f_2$ is a cocycle (resp.~coboundary) in $C^{n}(A,A\otimes A^{\sigma^{-1}})$.
    \end{enumerate}
    If the above conditions are satisfied, $f_3$ is also a cocycle (resp.~coboundary)
    in $C^{n}(A,A\otimes A^{\sigma^{-1}})$.
\end{lem}
\begin{proof} Take $l=k=1$ in \eqref{eq:eta}, then
\begin{align*}
 \eta^n(\diffb f\otimes 1\otimes 1)&=(\diffb f)_1 \otimes x \otimes
1 - (\diffb f)_2 \otimes 1 \otimes x + (\diffb f)_3 \otimes 1
\otimes 1 \\
&= \diffb_{\sigma^{-1}} f_1 \otimes x \otimes 1 -
\diffb_{\sigma^{-1}} f_2 \otimes 1 \otimes x + \diffb_{\sigma^{-1}}
f_3 \otimes 1 \otimes 1.
\end{align*}
It follows that $$(\diffb f)_1=(\sigma\otimes\id)(\diffb
f)(\sigma^{-1})^{\otimes n+1}= \diffb_{\sigma^{-1}} f_1, \, (\diffb
f)_2=(\id\otimes\,\sigma^{-1})(\diffb f)= \diffb_{\sigma^{-1}} f_2$$
and $(\diffb f)_3= \diffb_{\sigma^{-1}} f_3$.

So $f$ is a cocycle if and only if $f_1$ is a cocycle, if and only
if $f_2$ is a cocycle. If any one of $f$, $f_1$ and  $f_2$ is a
cocycle, then $f_3$ is also a cocycle.

If $f$ is a coboundary, say $f=\diffb g$, then
$f_1=(\sigma\otimes\id)(\diffb g) (\sigma^{-1})^{\otimes
n}=\diffb_{\sigma^{-1}}g_1$, $f_2=\diffb_{\sigma^{-1}}g_2$ and
$f_3=\diffb_{\sigma^{-1}}g_3$. Thus $f_1$, $f_2$ and $f_3$ are all
coboundaries.

If either $f_1=(\sigma\otimes\id)f(\sigma^{-1})^{\otimes n}$ or
$f_2=(\id\otimes\,\sigma^{-1})f$ is a coboundary, then $f$ is a
coboundary.
%
\end{proof}

\begin{thm}\label{thm:ore.hoch.cohomo.ses}
  Let $A$ be a projective $\kk$-algebra and $E = A[x; \sigma,\delta]$ be an Ore extension.
  Suppose that $A$ admits a finitely generated projective resolution as an $A^e$-module.
  Then for any $n\in\mathbb{N}$,
  \[
    0\xrightarrow{} H^{n}(A,E\otimes E)\xrightarrow{\partial}H^{n}(A,E\otimes E^{\sigma^{-1}})
    \xrightarrow {} H^{n+1}(E,E\otimes E)  \xrightarrow{}0
  \]
  is an exact sequence of $E^e$-modules.
\end{thm}
\begin{proof}
  Since $A$ admits a finitely generated projective resolution as an $A^e$-module,
  the two parallel arrows in \eqref{cd:quasiiso.ore.ext} are quasi-isomorphisms of $E^e$-module complexes.
  Thus the sequence \eqref{eq:long.e.s.hoch.cohomo.ore.ext} becomes
  \[
  \cdots\xrightarrow{}H^{n}(A,A\otimes A)\otimes \kk[x]^{\otimes 2}
  \xrightarrow{\tilde{\partial}}H^{n}(A,A\otimes A^{\sigma^{-1}})\otimes \kk[x]^{\otimes 2}
  \xrightarrow{} H^{n+1}(E,E\otimes E)\xrightarrow{} \cdots
  \]
  where $\tilde\partial$ is induced by $\partial$ and $\tilde\partial=H^n(\eta)$.

  It is sufficient to show $\tilde\partial$ is injective.

  Suppose that $\sum_{(l,k)} f^{l,k}\otimes\, x^{l}\otimes\, x^{k}$ is a cocycle
  in $C^{n}(A,A\otimes A)\otimes \kk[x]^{\otimes 2}$ such that $\tilde{\partial}
  \Big(\sum_{(l,k)} f^{l,k}\otimes x^{l}\otimes x^{k}+\im(\diffb^{n-1}\otimes\id^{\otimes 2})\Big)=0$. Then
  \begin{align}\label{eq:conn.homomorphism.monic}
    &\relphantom{=}\eta^{n}\bigg(\sum_{(l,k)} f^{l,k}\otimes\, x^{l}\otimes\,
    x^{k}\bigg)\notag\\
    &=\sum_{(l,k)} f^{l,k}_1\otimes\, x^{l+1}\otimes\,
    x^{k}-\sum_{(l,k)} f^{l,k}_2\otimes\, x^{l}\otimes\,
    x^{k+1}+\sum_{(l,k)} f^{l,k}_3\otimes\, x^{l}\otimes\,
    x^{k}\\
    &\,\in\im(\diffb_{\sigma^{-1}}^{n-1}\otimes\id^{\otimes 2}).\notag
  \end{align}
  Endow $\mathbb{N}^2$ with the lexicographical order from right to
  left, that is, $(a,b)>(c,d)$ if $b>d$ or ($b=d$, $a>c$). So the set
  consisting of all pairs $(l,k)$ such that $f^{l,k}\neq 0$ is a
  totally ordered set with respect to the order. Pick the greatest
  index $(l_0 ,k_0)$ and observe that $f^{l_0,k_0}_2\otimes
  x^{l_0}\otimes x^{k_0+1}$ is the unique term in
  \eqref{eq:conn.homomorphism.monic} containing $x^{l_0}\otimes
  x^{k_0+1}$ as its tensor factor. Therefore, $f^{l_0,k_0}_2$ is a
  coboundary and so is $f^{l_0,k_0}$. It follows that $\tilde{\partial}$ is
  injective.
\end{proof}

\section{Ore extensions preserve twisted Calabi-Yau property}\label{sec:ore.preserve.tcy}
In this section, we will show that the twisted Calabi-Yau property
is preserved by Ore extensions. First of all, recall the short
exact sequence \eqref{eq:ore.ses}. If $A$ admits a finitely
generated $A^e$-projective resolution of finite length, say
$P_{\cdot}$, and
\[\psi\colon E\otimes_AP_{\cdot}\otimes_A{}^{\sigma^{-1}}\!E \To E\otimes_AP_{\cdot}\otimes_AE\]
is a morphism lifting $\rho$, then $\cone(\psi)$ is a complex
 of finitely generated $E^e$-projective modules. Thus the
following proposition is concluded immediately.

\begin{prop}\label{prop:h.smooth.ore}
  Let $A$ be an algebra and $E=A[x;\sigma,\delta]$ be an Ore extension. If $A$ is homologically smooth, then so is $E$.
\end{prop}

Next, we consider the cohomology $H^*(E,E\otimes E)$.
\begin{prop}\label{prop:tw.cy.ore.ext}
  Let $A$ be a projective $\kk$-algebra and $E = A[x; \sigma,\delta]$ be an Ore extension. Suppose that
  \begin{enumerate}
    \item $A$ admits a finitely generated projective resolution as an $A^e$-module,
    \item $H^i(A,A\otimes A)=0$ unless $i=d$ for some $d\in\mathbb{N}$.
  \end{enumerate}
  Then $H^i(E,E\otimes E)=0$ unless $i=d+1$.

  Let $\omega$, $\omega'$ and $\Omega$ be the cohomology groups  $H^d(A,A\otimes
  A)$, $H^d(A,A\otimes A^{\sigma^{-1}})$
  and $H^{d+1}(E,E\otimes E)$, respectively.
  Then $\Omega \cong  \omega' \otimes \kk[x]$ and the $E^e$-module structure on
  $\omega'  \otimes \kk[x]$ is given as follows, for any $a\in A$, $[\tilde{f}]\in \omega' $, $k\in \mathbb{N}$,
  \begin{align}
  &a\triangleright([\tilde{f}]\otimes x^{k})
  =\sum_{i=0}^k q^k_i(a) [\tilde{f}]\otimes x^i,\label{eq:action1.homo.ore}\\
  &x\triangleright([\tilde{f}]\otimes x^{k})=[\tilde{f}]\otimes
  x^{k+1}, \label{eq:action2.homo.ore}\\
  &([\tilde{f}]\otimes x^{k})\triangleleft a
  =[\tilde{f}]a\otimes x^{k},\label{eq:action3.homo.ore}\\
  &([\tilde{f}]\otimes x^{k})\triangleleft x
  =[f_2]\otimes x^{k+1}-[f_3]\otimes x^k, \label{eq:action4.homo.ore}
  \end{align}
   where $f=(\sigma^{-1}\otimes\id)\tilde{f}(\sigma^{\otimes d})$, $f_2$ and $f_3$ are given by
   \eqref{eq:f2} and \eqref{eq:f3}.
\end{prop}
\begin{proof}
  Since $H^{i}(A,A\otimes A^{\sigma^{-1}})\cong H^{i}(A,A\otimes {}^{\sigma}\!A)$,
  by Theorem \ref{thm:ore.hoch.cohomo.ses}, $H^i(E,E\otimes E)=0$ for all $i\ne d+1$. And as $E^e$-modules,
  \[
  H^{d}(A,E\otimes E) \cong \omega \otimes\kk[x]^{\otimes 2},
  \]
  \[
  H^{d}(A,E\otimes E^{\sigma^{-1}})
  \cong \omega'   \otimes \kk[x]^{\otimes 2},
  \]
  where the $E^e$-module structure on $\omega'   \otimes \kk[x]^{\otimes 2}$ is given by
  \begin{alignat}{2}
    &x\cdot([\tilde{f}]\otimes x^{l}\!\otimes x^k)=[\tilde{f}]\otimes x^{l}\!\otimes x^{k+1},&\quad& \forall \,
    [\tilde{f}]\in \omega' ,\, k,l\in
    \mathbb{N},\label{eq:action1.ore}\\
    &a \cdot([\tilde{f}]\otimes x^{l}\!\otimes x^k)
    =\sum_{i=0}^k  q^k_i(a)  [\tilde{f}]\otimes x^l\otimes x^i, &\quad& \forall \,a\in A,\label{eq:action2.ore}\\
    &([\tilde{f}]\otimes x^{l}\!\otimes x^k)\cdot x=[\tilde{f}]\otimes x^{l+1}\!\otimes x^k,\label{eq:action3.ore}\\
    &([\tilde{f}]\otimes x^{l}\!\otimes x^k)\cdot a
    =\sum_{i=0}^l[\tilde{f}]p^l_i(a)\otimes x^{i}\otimes x^k,\label{eq:action4.ore}
  \end{alignat}
  and the $E^e$-module structure on $\omega \otimes \kk[x]^{\otimes 2}$ is given
  similarly.
  By the proof of Theorem \ref{thm:ore.hoch.cohomo.ses},
  \[
  0\xrightarrow{} \omega \otimes \kk[x]^{\otimes 2}\xrightarrow{\tilde{\partial}}
  \omega'  \otimes \kk[x]^{\otimes 2}\xrightarrow {} \Omega  \xrightarrow{}0
  \]
  is exact.

  To show $\Omega \cong  \omega' \otimes \kk[x]$, it suffices to show that
  $\omega' \otimes \kk[x]$ is  the cokernel of $\tilde\partial$.

  Now, for any cocycle
  $\tilde{f}\in C^{d}(A,A\otimes A^{\sigma^{-1}})$, let $f=(\sigma^{-1}\otimes\id)\tilde{f}(\sigma^{\otimes d})$.
  Then, by the definition of $\eta$, $f_1=\tilde{f}$  and
  \begin{equation}\label{eq:eqf123}
    \tilde{f}\otimes x^{l+1}\otimes x^{k}=f_2\otimes x^{l} \otimes
    x^{k+1}-f_3\otimes x^{l} \otimes x^{k} \pmod {\im \eta^{d}}.
  \end{equation}
  By Lemma \ref{lem:f1.f2.f3}, $f_2$, $f_3$ are also cocycles.
  If, in particular, $\tilde{f}$ is a coboundary, then so are $f_2$, $f_3$, and vice versa.
  It follows that for any $l$, $k\in\mathbb{N}$,
  \begin{equation}\label{eq:image.conn.homomorphism}
    \tilde{f}\otimes x^l\otimes x^k=\sum_{j=0}^lg_j\otimes 1 \otimes
    x^{j+k} \pmod {\im \eta^{d}}
  \end{equation}
  for some cocycles $g_j$ in $C^{d}(A,A\otimes A^{\sigma^{-1}})$, and $\tilde{f}$ is a coboundary
  if and only if all of $g_j$'s are coboundaries.

  Obviously, $f=0$ if and only if $f_1=0$. It follows from \eqref{eq:conn.homomorphism.monic}
   that $\sum_jg_j\otimes 1 \otimes x^j \in\im \eta^{d}$
   if and only if
  $g_j=0$ for all $j$. This implies that the cocycles $g_j$ in \eqref{eq:image.conn.homomorphism} are unique.
  Hence there exists a bijection
  \begin{align*}
    \Phi_1\colon(\omega'  \otimes \kk[x]^{\otimes 2})\big/\im\tilde{\partial}
    &\To \omega' \otimes \kk[x]\\
    [\tilde{f}]\otimes x^l\otimes x^k+\im\tilde{\partial}&\longmapsto\sum_{j=0}^l[g_j]\otimes
    x^{j+k}.
  \end{align*}

  Therefore, $\Omega \cong \omega' \otimes \kk[x]$.
  It follows from \eqref{eq:action1.ore}, \eqref{eq:action2.ore}, \eqref{eq:action4.ore}
  that the induced $E^e$-module structure on $ \omega'  \otimes\kk[x]$ satisfies
  \eqref{eq:action1.homo.ore}, \eqref{eq:action2.homo.ore},
  \eqref{eq:action3.homo.ore}. By \eqref{eq:eqf123},
  \begin{equation}\label{eq:eq.x.right}
  \Phi_1([f_1]\otimes x\otimes 1+\im\tilde{\partial})
  =[f_2]\otimes x-[f_3]\otimes 1.
  \end{equation}
  Then it follows from \eqref{eq:action2.homo.ore} and \eqref{eq:action3.ore}
  that $([\tilde{f}]\otimes x^k)\triangleleft
  x =[f_2]\otimes x^{k+1}-[f_3]\otimes x^k$, i.e., \eqref{eq:action4.homo.ore} holds.
\end{proof}

\begin{thm}\label{thm:nakayama.auto.ore}
  Let $A$ be a projective $\kk$-algebra and $E=A[x;\sigma,\delta]$ be an Ore extension.
  Suppose that $A$ is $\nu$-twisted Calabi-Yau of dimension $d$.
  Then $E$ is twisted Calabi-Yau of dimension $d+1$ and the Nakayama automorphism
  $\nu'$ of $E$
  satisfies that $\nu'|_{A}=\sigma^{-1}\nu$ and $\nu'(x)=ux+b$ with $u$, $b\in A$ and $u$ invertible.
\end{thm}
\begin{proof}
We still use $\omega$, $\omega'$ and $\Omega$ as above.
  As $\omega' \cong {}^{\sigma} \omega$ and $\omega \cong
A^{\nu}$, we may fix a bimodule isomorphism $\vphi\colon \omega' \to
{}^{\sigma}\!A^{\nu}$.

It follows from Proposition \ref{prop:tw.cy.ore.ext} that $\Omega
\cong  \omega' \otimes \kk[x]
\cong{}^{\sigma}\!A^{\nu}\otimes\kk[x]$. The $E\otimes
A^{\op}$-module structure on ${}^{\sigma}\!A^{\nu}\otimes \kk[x]$ is
induced from  \eqref{eq:action1.homo.ore},
\eqref{eq:action2.homo.ore} and  \eqref{eq:action3.homo.ore}. Let us
prove ${}^{\sigma}\!A^{\nu}\otimes \kk[x]\cong E^{\,\sigma^{-1}\nu}$
as $E\otimes A^{\op}$-modules.

In fact, the composite
\[
  {}^{\sigma}\!A^{\nu}\otimes \kk[x]\xrightarrow[\quad]
  {\sigma^{-1}\otimes\id}A^{\sigma^{-1}\nu}\!\otimes \kk[x]
  \xrightarrow[\quad]{\tau}\kk[x]\otimes
  A^{\sigma^{-1}\nu}\xrightarrow[\quad]{\mu}E^{\,\sigma^{-1}\nu},
\]
denoted by $\Phi_3$, is an isomorphism of $E\otimes
A^{\op}$-modules.

Clearly, $\Phi_3$ is bijective. For any $a',a\in A$ and $k\in
\mathbb{N}$,
\begin{align*}
  \Phi_3((a'\otimes x^k)\triangleleft a)&=\Phi_3(a' \nu(a)\otimes x^k)=x^k \sigma^{-1}(a') \sigma^{-1}\nu(a) \\
  &=\Phi_3(a'\otimes x^k)\cdot a,\\
  \Phi_3(x\triangleright(a'\otimes x^k))&=\Phi_3(a'\otimes
  x^{k+1})=x^{k+1}\sigma^{-1}(a')=xx^{k}\sigma^{-1}(a')\\
  &=x\cdot\Phi_3(a'\otimes x^k).
\end{align*}
Recall the maps  $q^k_i\colon A\to A$ in \eqref{eq:ore.left} such
that $ax^{k}=\sum_{i=0}^{k}x^iq^k_i(a)$,
\begin{align*}
\Phi_3(a\triangleright(a'\otimes x^k))&=\Phi_3\bigg(\sum_{i=0}^k\sigma (q^k_i(a))a'\otimes x^i\bigg)\\
&=\sum_{i=0}^k x^iq^k_i(a)\sigma^{-1}(a')=ax^k\sigma^{-1}(a')\\
&=a\cdot\Phi_3(a'\otimes x^k).
\end{align*}

So $\Omega \cong E^{\,\sigma^{-1}\nu}$ as $E\otimes
A^{\op}$-modules. There exists an endomorphism $\nu'$ of $E$ such
that $\Omega \cong E^{\,\nu'}$ as $E^e$-modules and
$\nu'|_A=\sigma^{-1}\nu$. In such a way,  $\Phi_3$ is indeed an
isomorphism of $E^e$-modules.

Now we try to decide $\nu'(x)$. Let $\Phi_2=\vphi\otimes\id\colon
\omega' \otimes \kk[x] \to {}^{\sigma}\!A^{\nu} \otimes \kk[x]$.
Since $\omega' \cong {}^{\sigma}\!A^{\nu}$ via $\vphi$, there exists
a cocycle $\tilde{f}\in C^{d}(A,A\otimes A^{\sigma^{-1}})$ such that
$$\Phi_2\Phi_1([\tilde{f}]\otimes 1 \otimes 1 +\im\tilde{\partial})=1_A\otimes 1.$$ Define
$f$, $h\in C^{d}(A,A\otimes A)$ by
$f=(\sigma^{-1}\otimes\id)\tilde{f}(\sigma^{\otimes d})$ and
$h=(\id\otimes\,\sigma)\tilde{f}$, respectively. Clearly,
$\tilde{f}=f_1=h_2$. Thus $f$ and $h$ are both cocycles. Then
\begin{align*}
  \nu'(x)&=1_E\cdot x=\Phi_3\Phi_2\Phi_1([\tilde{f}]\otimes 1 \otimes 1+\im\tilde{\partial})\cdot x\\
  &=\Phi_3\Phi_2\Phi_1([\tilde{f}]\otimes x \otimes 1+\im\tilde{\partial}) & &\mbox{by \eqref{eq:action3.ore}} \\
  &=\Phi_3\Phi_2\Phi_1([f_1]\otimes x \otimes 1+\im\tilde{\partial})\\
  &=\Phi_3\Phi_2([f_2]\otimes x)-\Phi_3\Phi_2([f_3]\otimes 1) & & \mbox{by \eqref{eq:eq.x.right}} \\
  &=\Phi_3(\vphi([f_2])\otimes x)-\Phi_3(\vphi([f_3])\otimes 1)  \\
  &=x\sigma^{-1}\vphi([f_2])-\sigma^{-1}\vphi([f_3])  \\
  &=\vphi([f_2])x+\delta\sigma^{-1}\vphi([f_2])-\sigma^{-1}\vphi([f_3]).
\end{align*}
Let $u=\vphi([f_2])$ and
$b=\delta\sigma^{-1}\vphi([f_2])-\sigma^{-1}\vphi([f_3])$.  Then
$\nu'(x)=ux+b$.

On the other hand,
\begin{align*}
  x&=x\cdot\Phi_3\Phi_2\Phi_1([\tilde{f}]\otimes 1\otimes 1+\im\tilde{\partial})=\Phi_3\Phi_2\Phi_1([h_2]
  \otimes 1\otimes x+\im\tilde{\partial})\\
  &=\Phi_3\Phi_2\Phi_1([h_1]\otimes x\otimes 1+\im\tilde{\partial})+\Phi_3\Phi_2\Phi_1([h_3]\otimes 1
  \otimes 1+\im\tilde{\partial})\\
  &=\Phi_3\Phi_2\Phi_1([h_1]\otimes 1 \otimes 1+\im\tilde{\partial})\cdot x+\Phi_3\Phi_2\Phi_1([h_3]
  \otimes 1\otimes 1+\im\tilde{\partial})\\
  &=\sigma^{-1}\vphi([h_1])\cdot x+\sigma^{-1}\vphi([h_3]).
\end{align*}
Let $v=\sigma^{-1}\vphi([h_1])$, $c=\sigma^{-1}\vphi([h_3])$. Then
\[x=v\cdot x+c=v(ux+b)+c=vux+vb+c,\]
which implies $vu=1_A$ and $vb+c=0$.

Since $\nu'|_A$ is an automorphism and $u$ is left invertible,
$x\in\im\nu'$, namely, $\nu'$ is surjective. Suppose that
$\nu'(\sum_{i=0}^nx^ia_i)=0$. Then develop
$\nu'(\sum_{i=0}^nx^ia_i)=\sum_{i=0}^n(ux+b)^i\sigma^{-1}\nu(a_i)$
to the form $\sum_{i=0}^nx^ia'_i$. It is easy to show that the
leading term is
$x^n\sigma^{-n}(u)\cdots\sigma^{-2}(u)\sigma^{-1}(u)\sigma^{-1}\nu(a_n)$.
So the coefficient is zero. Since $u$ is left invertible and
$\sigma$, $\nu$ are automorphisms, $a_n=0$. Consequently, $\nu'$ is
injective.

Finally, we prove that $u$ is also right invertible. In fact, for
any $a\in A$, $xa=\sigma(a)x+\delta(a)$. Under the action of $\nu'$,
\begin{align*}
  &\relphantom{=}(ux+b)\sigma^{-1}\nu(a)\\
  &=u(\nu(a)x+\delta\sigma^{-1}\nu(a))+b\sigma^{-1}\nu(a)\\
  &=\sigma^{-1}\nu\sigma(a)(ux+b)+\sigma^{-1}\nu\delta(a).
\end{align*}
Comparing the coefficients of $x$, we have
$\sigma^{-1}\nu\sigma(a)u=u\nu(a)$ for any $a\in A$. In particular,
let $a=\sigma^{-1}\nu^{-1}\sigma(v)$ and so $u$ is also right
invertible.

Therefore, by Propositions \ref{prop:h.smooth.ore}, \ref{prop:tw.cy.ore.ext}, $E$ is twisted Calabi-Yau of dimension
$d+1$ and the Nakayama automorphism $\nu'$ satisfies the required
conditions.
\end{proof}

\begin{rk}
  By the definition of $\eta$ in \eqref{cd:quasiiso.ore.ext}, $f_1=f_2$ if $\sigma=\id$, and $f_3=0$ if $\delta=0$.
  Thus $\nu'(x)=x+b$ if $\sigma=\id$, and $\nu'(x)=ux$ if $\delta=0$.
\end{rk}


\section{Applications}\label{sec:app}
One  motivation of studying the twisted Calabi-Yau property of Ore
extensions is studying the right coideal subalgebras of the positive
Borel part of a quantized enveloping algebra and computing their
Nakayama automorphisms \cite{LW:tcy.rcs.q.enve.alge} by the first
named and the third named authors. Such algebras can be obtained by
iterated Ore extensions. In \cite{LW:tcy.rcs.q.enve.alge},  a class
of right coideal subalgebras (quantum homogeneous spaces)
$C\subseteq U_q(\mathfrak{g})$ is proved to be twisted Calabi-Yau,
and the Nakayama automorphisms are given explicitly in some cases.

In this section,  the base ring $\kk$ is assumed to be a field.

\subsection{Quantum affine spaces}
As  stated in Section \ref{sec:def.ore}, multi-parameter quantum
affine $n$-spaces $\mathcal{O}_{\mathbf{q}}(\kk^n)$ can be obtained
by iterated Ore extensions. Their Nakayama automorphisms can be
computed by using  Theorem \ref{thm:nakayama.auto.ore}. Of course,
all the results in this subsection are known and can be deduced in
some other way.

Let $n\ge 1$ and $\mathbf{q}$ be a matrix $(q_{ij})_{n\times n}$ whose entries are in $\kk$
satisfying $q_{ii}=1$ and $q_{ij}q_{ji}=1$ for all $1\le i,j\le n$.
The quantum affine $n$-space $\mathcal{O}_{\mathbf{q}}(\kk^n)$ is defined to
be a $\kk$-algebra generated by $x_1,\ldots,x_n$ with the relations $x_jx_i=q_{ij}x_ix_j$ for all $1\le i,j\le n$.
\begin{prop}
The quantum affine $n$-space $\mathcal{O}_{\mathbf{q}}(\kk^n)$ is
twisted Calabi-Yau of dimension $n$, whose Nakayama automorphism
$\nu$ sends $x_i$ to $(\prod_{j=1}^nq_{ji})x_i$.
\end{prop}
\begin{proof}
If $n=1$, $\mathcal{O}_{\mathbf{q}}(\kk)=\kk[x_1]$. The conclusion
is true.

If $n>1$, we assume the conclusion holds for $n-1$. Let
$\mathbf{q'}$ be an $(n-1)\times (n-1)$ matrix obtained by deleting
the $n^{\mathrm{th}}$ row and the $n^{\mathrm{th}}$ column of
$\mathbf{q}$, and $\mathbf{q''}$ by deleting the first row and the
first column of $\mathbf{q}$. Now consider the following two quantum
$(n-1)$-spaces
  \begin{align*}
    \mathcal{O}_{\mathbf{q'}}(\kk^{n-1})&=\kk\langle x_1,\ldots,x_{n-1}\mid x_jx_i=q_{ij}x_ix_j,\,1\le i,j\le n-1\rangle,\\
    \mathcal{O}_{\mathbf{q''}}(\kk^{n-1})&=\kk\langle x_2,\ldots,x_{n}\mid x_jx_i=q_{ij}x_ix_j,\,2\le i,j\le n\rangle.
  \end{align*}
Clearly,
$\mathcal{O}_{\mathbf{q}}(\kk^n)=\mathcal{O}_{\mathbf{q'}}(\kk^{n-1})[x_n;\sigma']$
where $\sigma'(x_i)=q_{in}x_i$ for $1\le i\le n-1$, and
$\mathcal{O}_{\mathbf{q}}(\kk^n)=\mathcal{O}_{\mathbf{q''}}(\kk^{n-1})[x_1;\sigma'']$
where $\sigma''(x_i)=q_{i1}x_i$ for $2\le i\le n$.

By the inductive hypothesis, $\mathcal{O}_{\mathbf{q'}}(\kk^{n-1})$
and $\mathcal{O}_{\mathbf{q''}}(\kk^{n-1})$ are both twisted
Calabi-Yau of dimension $n-1$ and their Nakayama automorphisms
$\nu'$, $\nu''$ are given by
  \begin{align*}
    \nu'(x_i)&=\Big(\prod_{j=1}^{n-1}q_{ji}\Big)x_i,\quad 1\le i\le n-1, \\
    \nu''(x_i)&=\Big(\prod_{j=2}^{n}q_{ji}\Big)x_i,\quad 2\le i\le n,
  \end{align*}
  respectively.

  Since the invertible elements in $\mathcal{O}_{\mathbf{q}}(\kk^n)$ are those nonzero scalars in $\kk$,
  the identity map is the only inner automorphism of $\mathcal{O}_{\mathbf{q}}(\kk^n)$.
  By Theorem \ref{thm:nakayama.auto.ore}, $\mathcal{O}_{\mathbf{q}}(\kk^n)$ is twisted Calabi-Yau of dimension $n$
  whose Nakayama automorphism $\nu$ satisfies
  \begin{align*}
    \nu(x_i)&=\sigma'^{-1}\bigg(\Big(\prod_{j=1}^{n-1}q_{ji}\Big)x_i\bigg)=\Big(\prod_{j=1}^{n}q_{ji}\Big)x_i,\quad 1\le i\le n-1, \\
    \nu(x_i)&=\sigma''^{-1}\bigg(\Big(\prod_{j=2}^{n}q_{ji}\Big)x_i\bigg)=\Big(\prod_{j=1}^{n}q_{ji}\Big)x_i,\quad 2\le i\le n.
  \end{align*}
  So $\nu(x_i)=(\prod_{j=1}^{n}q_{ji})x_i$ for $1\le i\le n$.

  Therefore, the proposition holds for all $n\ge 1$.
\end{proof}
\begin{rk}
  The same method can be applied to Weyl algebras $A_n(\kk)$, $n\ge 1$. As a consequence, Weyl algebra $A_n(\kk)$ is
  Calabi-Yau of dimension $2n$.
\end{rk}

\subsection{A $3$-dimensional AS-regular algebra}

Let $A$ be generated by $x$, $y$, $z$ with three relations
\[yx-xy-x^2,\,zx-xz,\,zy-yz-2xz.\]
Then $A$ is a 3-dimensional AS-regular algebra.

Let $B=\kk\langle x,y\rangle/(yx-xy-x^2)$ be the Jordan plane, which
is an AS-regular algebra of dimension $2$. Obviously, $B=\kk[x][y;
\delta_1]$ with $\delta_1 (x)=x^2$. It follows that $B$ is twisted
Calabi-Yau, but not Calabi-Yau, with the  Nakayama automorphism
given by $\nu(x)=x$ and $\nu(y)=2x+y$.

On one hand, $A=B[z;\nu]$ is an Ore extension of Jordan plane. Then
$A$ is twisted Calabi-Yau with the Nakayama automorphism $\nu'$ such
that $\nu'(x)=x$ and $\nu'(y)=y$.

On the other hand, $A=\kk[x,z][y;\delta]$ where $\delta$ is given by
$\delta(x)=x^2$ and $\delta(z)=-2xz$. So, $\nu'(z)=z$.

It follows that $A$ is Calabi-Yau, which was proved by Berger and
Pichereau \cite{Berger-Pichereau:cy.deform.poisson.alg}.

%
%
%

\subsection{A class of AS-regular algebras of dimension \protect\(5\)}
Classifying quantum projective spaces $\mathbb{P}^n$ \!---
noncommutative analogues of projective $n$-spaces, is one of the
most important questions in noncommutative projective algebraic
geometry. An algebraic approach to construct a quantum
$\mathbb{P}^n$ is to form the noncommutative projective scheme
$\mathrm{Proj}\:A$ \cite{Artin-Zhang:noncomm.proj}, where $A$ is a
noetherian connected graded AS-regular algebra of global dimension
$n+1$. So the question turns out to be the classification of
AS-regular algebras.

Recently, the second named and the third named authors tried to
classify quantum $\mathbb{P}^4$s. In \cite{Wang-Wu:5.dim.reg.alg},
AS-regular algebras of dimension 5, generated by two generators of
degree $1$ with three generating relations of degree $4$, are
classified under some generic condition. There are nine types such
AS-regular algebras in the classification list. Among them, algebras
$\D$ and $\G$ can be realized by iterated Ore extensions
(\cite[Proposition 5.7 and Theorem 5.8]{Wang-Wu:5.dim.reg.alg}).

In this subsection, we compute the Nakayama automorphisms of these
two types of algebras. Assume $\kk$ is a field of characteristic
zero. The algebras $\D$ and $\G$ are of the form $\kk\langle
x,y\rangle/(r_1,r_2,r_3)$.

For algebra $\D$,
\begin{align*}
  r_1&=x^{3}y+p x^{2}yx+q xyx^{2}-p(2p^2+q)yx^{3},\\
  r_2&=x^{2}y^{2}-p(p^2+q)yxyx-q^2 y^{2}x^{2}+(q- p^2) xy^{2}x+(q-p^2)yx^{2}y,\\
  r_3&=xy^{3}+p yxy^{2}+q y^{2}xy-p(2p^2+q)y^{3}x,
\end{align*}
where $p$, $q \in \kk \setminus \{0\}$ and $2p^{4}-p^{2}q+q^{2}=0$.


For algebra $\G$,
\begin{align*}
  r_1&=x^{3}y+p x^{2}yx+q xyx^{2}+syx^{3},\\
  r_2&=x^{2}y^{2}+l_2 xyxy+l_3 yxyx+l_4 y^{2}x^{2}+l_5 xy^{2}x+l_5yx^{2}y,\\
  r_3&=xy^{3}+p yxy^{2}+q y^{2}xy+sy^{3}x,
\end{align*}
where
\[
  l_2=-\frac{s^2(qs-g)}{g(qs+g)},\:l_3=s-\frac{pg (p s-q^2)}{q (q s+g)},
  l_4=-\frac{g^2}{s^2},\:l_5=\frac{p s^2+q g}{q s+g},
\]
with $p$, $q$, $s$, $g\in\kk\setminus\{0\}$,
$ps^3g+qsg^2+s^5+g^3=0$, $p^3s=q^3$, $ps\neq q^2$, $q^2 s^2 \neq g^2$
and $s^5+g^3\neq 0$.

It is proved that the algebras $\D$ and $\G$ can be obtained as an
iterated Ore extension by a unified process \cite[subsection
5.2]{Wang-Wu:5.dim.reg.alg}. We give a sketch of the process here.

Let $A=\kk[y]$ with $\deg y
=1$. Let $a$, $b\in\kk$ satisfy $ab(a+b)(a^2+b^2)(a^3-b^3)\neq 0$.

Define $A_1=A[z_1; \sigma_1]$ to be the graded Ore extension of $A$ with $\deg z_1=3$, where
\[\sigma_1(y)=a y.\]

Define $A_2=A_1[z_2; \sigma_2, \delta_2]$ to be the graded Ore
extension of $A_1$ with $\deg z_2=2$, where
\begin{alignat*}{2}
  \sigma_2(y)&=b y, & \quad \sigma_2(z_1)&=a z_1,\\
  \delta_2(y)&=z_1,  & \quad \delta_2(z_1)&=0.
\end{alignat*}

Define $A_3=A_2[z_3; \sigma_3, \delta_3]$ to be the graded Ore extension
of $A_2$ with $\deg z_3=3$, where
\begin{alignat*}{3}
  \sigma_3(y)&=a^{-1}b^3  y, & \quad  \sigma_3(z_1)&=b^3 z_1, &\quad
  \sigma_3(z_2)&=az_2,\\
  \delta_3(y)&=z_2^2, & \quad
  \delta_3(z_1)&=(a-b) z_2^3, &\quad
  \delta_3(z_2)&=0.
\end{alignat*}

Define $A_4=A_3[x; \sigma_4,\delta_4]$ to be the graded Ore extension of $A_3$ with $\deg x=1$, where
\begin{alignat*}{4}
\sigma_4(y)&=a^{-1}b^2  y, &\quad \sigma_4(z_1)&=a^{-1}b^3  z_1,
&\quad \sigma_4(z_2)&=b z_2, &\quad \sigma_4(z_3)&=az_3,\\
\delta_4(y)&=z_2, &\quad \delta_4(z_1)&=\frac{a^3-b^3}{a(a+b)}
z_2^2, &\quad \delta_4(z_2)&=\frac{a^3-b^3}{a(a+b)}z_3, &\quad
\delta_4(z_3)&=0.
\end{alignat*}

Let $a=p^{-3}q^2$, $b=-p^{-1}q$, then $A_4\cong\D$. Let $a=s^2g^{-1}$, $b=-p^{-1}q$, then $A_4\cong\G$.
Both isomorphisms send the indeterminants $x$, $y$ in $A_4$ to the generators $x$, $y$ of $\D$ and $\G$, respectively.

Now let us compute the graded Nakayama automorphism $\nu$ of $A_4$.

By Theorem \ref{thm:nakayama.auto.ore},
$\nu(y)=\sigma_4^{-1}\sigma_3^{-1}\sigma_2^{-1}\sigma_1^{-1}(y)=ab^{-6}y$.

Observe that $A_4$ can be also obtained as an iterated Ore extension
along the opposite direction, that is, adding $z_3$, $z_2$, $z_1$,
$y$ to $\kk[x]$ successively. The corresponding automorphisms and
derivations are determined by $\sigma_i$ and $\delta_i$ ($1\le i\le
4$). We do not give their concrete expressions but only the result
$\nu(x)=a^{-1}b^6x$.

Return to the algebras $\D$ and $\G$. For $\D$,
$a^{-1}b^6=p^{3}q^{-2}p^{-6}q^6=p^{-3}q^4$,
and the Nakayama automorphism $\nu$ is given by
\[\nu(x)=p^{-3}q^4x,\, \nu(y)=p^{3}q^{-4}y.\]
For $\G$,
$a^{-1}b^6=s^{-2}gp^{-6}q^6=g$,
and the Nakayama automorphism $\nu$ is given by
\[\nu(x)=gx,\; \nu(y)=g^{-1}y.\]

Thus we have
\begin{thm}
\begin{enumerate}
\item The algebra $\D$ is twisted Calabi-Yau with the Nakayama automorphism
$\nu$ given by
\[\nu(x)=p^{-3}q^4x,\, \nu(y)=p^{3}q^{-4}y.\]
And $\D$ is Calabi-Yau if and only if that $p$, $q$ satisfy the
system of equations
\[\begin{cases}
  p^3=q^{4},\\
  2p^{4}-p^{2}q+q^{2}=0.
\end{cases}\]
 \item The algebra $\G$ is twisted Calabi-Yau with the Nakayama automorphism
$\nu$ given by \[\nu(x)=gx,\; \nu(y)=g^{-1}y.\] And $\G$ is
Calabi-Yau if and only if that $g=1$.
\end{enumerate}
\end{thm}

\section*{Acknowledgments}
This research is supported by the NSFC (key project 10731070), and
STCSM (Science and Technology Committee, Shanghai Municipality,
project 11XD1400500), and a training program for innovative talents
of key disciplines, Fudan University.


\end{document}